\documentclass[reqno,report]{amsart}
\usepackage{latexsym}
\usepackage{amssymb}
\usepackage{amsmath}
\usepackage{amsfonts}
\usepackage{amsthm}
\usepackage{amscd}
\usepackage{texdraw}
\usepackage{color}
\usepackage[centertags]{amsmath}

\newlength{\defbaselineskip}
\newcommand{\setlinespacing}[1]%
           {\setlength{\baselineskip}{#1 \defbaselineskip}}

\theoremstyle{plain}
\newtheorem{theorem}{Theorem}
\newtheorem{lemma}{Lemma}

\newtheorem{corollary}{Corollary}
\newtheorem{definition}{Definition}

{\allowdisplaybreaks
\begin{document}
\title[Symmetry and inverse-closedness of some $p$-Beurling algebras]{Symmetry and inverse-closedness of some $p$-Beurling algebras}

\author[P. A. Dabhi]{Prakash A. Dabhi}
\author[K. B. Solanki]{Karishman B. Solanki}
\address{Department of Mathematics, Institute of Infrastructure Technology Research and Management(IITRAM), Maninagar (East), Ahmedabad - 380026, Gujarat, India}

\email{lightatinfinite@gmail.com}
\email{karishsolanki002@gmail.com}
\thanks{The authors are very much grateful to Professor W. \.Zelazko for his help by sending his book. The first author would like to thank SERB, India, for the MATRICS grant no. MTR/2019/000162. The second author gratefully acknowledges Junior Research Fellowship (NET) from CSIR, India.}



\date{}

\dedicatory{}

\commby{}

\begin{abstract}
Let $(G,d)$ be a metric space with the counting measure $\mu$ satisfying some growth conditions. Let $\omega(x,y)=(1+d(x,y))^\delta$ for some $0<\delta\leq1$. Let $0<p\leq1$. Let $\mathcal A_{p\omega}$ be the collection of kernels $K$ on $G\times G$ satisfying $\max\{\sup_x\sum_y |K(x,y)|^p\omega(x,y)^p, \sup_y\sum_x |K(x,y)|^p\omega(x,y)^p\}<\infty$. Each $K \in \mathcal A_{p\omega}$ defines a bounded linear operator on $\ell^2(G)$. If in addition, $\omega$ satisfies the weak growth condition, then we show that $\mathcal A_{p\omega}$ is inverse closed in $B(\ell^2(G))$. We shall also discuss inverse-closedness of $p$-Banach algebra of infinite matrices over $\mathbb Z^d$ and the $p$-Banach algebra of weighted $p$-summable sequences over $\mathbb Z^{2d}$ with the twisted convolution. In order to show these results, we prove Hulanicki's lemma and Barnes' lemma for $p$-Banach algebras.
\end{abstract}
\maketitle

\section{Introduction}
Let $0<p\leq1$, and let $\mathcal{A}$ be an algebra. A mapping $\|\cdot\| : \mathcal{A} \to [0,\infty)$ is a \emph{$p$-norm} \cite{Ze} on $\mathcal{A}$ if the following conditions hold for all $x,y\in \mathcal{A}$ and $\alpha \in \mathbb{C}$.
\begin{enumerate}
\item $\|x\|=0$ if and only if $x=0$;
\item $\|x+y\|\leq\|x\|+\|y\|$;
\item $\|\alpha x\|= |\alpha|^p \|x\|$;
\item $\|xy\|\leq\|x\|\|y\|$.
\end{enumerate}
If $\mathcal{A}$ is complete in the $p$-norm, then $(\mathcal{A},\|\cdot\|)$ is a \emph{$p$-Banach algebra} \cite{Ze}. When $p=1$, $\mathcal A$ is a Banach algebra.

A $p$-normed (Banach) $\ast$-algebra is a $p$-normed (Banach) algebra along with an isometric involution $\ast$. A \emph{$p$-$C^\ast$-algebra} is a $C^\ast$-algebra $(\mathcal A,\|\cdot\|)$ with the \emph{$p$-$C^\ast$-norm} $|x|=\|x\|^p\;(x \in \mathcal A)$. Let $\mathcal{A}$ be a $p$-Banach algebra with unit $e$, and let $x\in\mathcal{A}$. The set $\sigma_\mathcal{A}(x)=\{\lambda\in\mathbb{C}: \lambda e-x \ \text{is not invertible in}\ \mathcal{A} \}$ is the \emph{spectrum} of $x$ in $\mathcal{A}$ and the number $r_\mathcal{A}(x)=\sup\{|\lambda|^p:\lambda\in\sigma_\mathcal{A}(x)\}$ is the \emph{spectral radius} of $x$. The spectral radius formula gives $r_\mathcal{A}(x)=\lim_{n\to\infty}\|x^n\|^\frac{1}{n}$ \cite{Ze}. We shall just write $\sigma(x)$ and $r(x)$ when the algebra in consideration is clear.

Let $\mathcal{A}$ be a commutative $p$-Banach algebra. A nonzero linear map $\varphi:\mathcal A \to \mathbb C$ satisfying $\varphi(ab)=\varphi(a)\varphi(b)\;(a,b \in \mathcal A)$ is a \emph{complex homomorphism} on $\mathcal A$. Let $\Delta(\mathcal A)$ be the collection of all complex homomorphisms on $\mathcal A$. For $a \in \mathcal A$, let $\widehat a:\Delta(\mathcal A)\to \mathbb C$ be $\widehat a(\varphi)=\varphi(a)\;(\varphi\in \Delta(\mathcal A))$. The smallest topology on $\Delta(\mathcal A)$ making each $\widehat a$, $a\in \mathcal A$, continuous is the \emph{Gel'fand topology} on $\Delta(\mathcal A)$ and $\Delta(\mathcal A)$ with the Gel'fand topology is the \emph{Gel'fand space} of $\mathcal{A}$. For more details on it refer \cite{Ge, Ze}.

Let $\mathcal{H}$ be a Hilbert space. Then $B(\mathcal{H})$, the collection of all bounded linear operators on $\mathcal{H}$, is a $C^\ast$-algebra with the operator norm $\|T\|_{op}=\sup\{\|T(x)\|:x\in\mathcal{H}, \|x\|\leq1\}$ for all $T \in B(\mathcal H)$.

Note that for given $0<p\leq1$ and a normed (Banach) algebra $\mathcal{A}$ with norm $\|\cdot\|$, we may consider the $p$-norm, $\|\cdot\|_p$, on $\mathcal{A}$ given $\|x\|_p=\|x\|^p \ (x\in\mathcal{A})$ making $\mathcal{A}$ a $p$-normed ($p$-Banach) algebra without changing the topology of $\mathcal{A}$. The fact that $(a+b)^p\leq a^p+b^p$ for all $a,b\in [0,\infty)$ and $0<p\leq 1$ will be used here and many times in this paper. All algebras considered here are complex algebras, i.e., over the complex field $\mathbb{C}$.

In \cite{Hu}, Hulanicki proved that if $\mathcal{A}$ is a Banach $\ast$-algebra, $S$ is a subalgebra of $\mathcal{A}$ (not necessarily closed) and if $T:\mathcal{A}\to B(\mathcal{H})$ is a faithful representation for some Hilbert space $\mathcal{H}$ such that $\|T_x\|=\lim_{n\to\infty}\|x^n\|^\frac{1}{n}$ for all $x=x^\ast\in S$, then $\sigma_\mathcal{A}(x)=\sigma(T_x)$ for all $x=x^\ast\in S$. The corrected proof of this theorem can be found in \cite{Fe}. We prove this result for $p$-Banach algebras.

Let $(G,d)$ be a metric space, and let $\mu$ be a measure on $G$. For $\delta>0$, let $\Gamma[\delta]=\{(x,y) \in G \times G : d(x,y)\leq\delta \}$, and for $x\in G$, let $\Gamma_x[\delta]=\{y\in G: d(x,y)\leq\delta \}$. Assume that there are constants $C>0, b>0$ such that $\mu(\Gamma_x[\delta])\leq C\delta^b$ for all $x\in G$ and $\delta>0$. Let $0<\delta\leq1$ be fixed, and let $\omega:G\times G \to [1,\infty)$ be $\omega(x,y)=(1+d(x,y))^\delta$. Let $0<p\leq 1$, and let $\mathcal A_{p\omega}$ be the collection all complex valued measurable functions $K=K(x,y)$ on $G \times G$ such that $$\|K\|_{p\omega}=\max\Big\{\sup_x \int_G |K(x,y)|^p\omega(x,y)^pd\mu(y),\sup_y \int_G|K(x,y)|^p\omega(x,y)^pd\mu(x)\Big\}<\infty.$$ Note that $\mathcal A_{1\omega}$ is a Banach $\ast$-algebra with the above norm, the convolution multiplication $$(K\star J)(x,y)=\int_G K(x,z)J(z,y)d\mu(z)$$ and the involution $K \mapsto K^\ast$, where $K^\ast(x,y)=\overline{K(y,x)}$. By \cite{Ba}, $K$ defines a bounded linear operator $K_2$ on $L^2(G)$ by $K_2(f)(x)=\int_G f(y)K(x,y)d\mu(y)$ for all $f \in L^2(G)$. Barnes proved in \cite{Ba} that the spectrum of $K$ as an element of $\mathcal A_{1\omega}$ is same as the spectrum of $K_2$ in $B(L^2(G))$.

Let $0<p<1$, and let $K, J \in \mathcal A_{p\omega}$. Then $|\int_G K(x,z)J(z,y) d\mu(z)|^p$ may not be smaller than $\int_G|K(x,z)|^p|J(z,y)|^p d\mu(z)$. So, if we want this inequality to remain true or if we want $\mathcal A_{p\omega}$ to be an algebra, then we should take $\mu$ to be the counting measure. One more reason for taking $\mu$ to be the counting measure on $G$ is as follows. Let $G$ be a locally compact group with the Haar measure $\mu$, let $\omega$ be a measurable weight on $G$ and let $L^p(G,\omega)$ be the collection of all measurable functions on $G$ satisfying $\int_G |f|^p\omega^p d\mu<\infty$. Then by \cite{Ze}, $L^p(G)$ is closed under convolution if and only if $G$ is a discrete group and by \cite{Bh}, $L^p(G,\omega)$ is closed under convolution if and only if $G$ is a discrete group.

So, we shall consider the counting measure $\mu$ on a metric space $G$. In this case, $\mathcal A_{p\omega}$, $0<p\leq 1$, will be the collection of all functions $K:G \times G \to \mathbb C$ satisfying $$\|K\|_{p\omega}=\max\Big\{\sup_x \sum_y |K(x,y)|^p\omega(x,y)^p, \sup_y \sum_x |K(x,y)|^p\omega(x,y)^p\Big\}<\infty.$$ Then $\mathcal A_{p\omega}$ is a $p$-Banach $\ast$-algebra with the above norm, the convolution $$(K\star J)(x,y)=\sum_z K(x,z)J(z,y)\quad(K, J \in \mathcal A_{p\omega}, (x,y)\in G \times G)$$ and the involution $K^\ast(x,y)=\overline{K(y,x)}$. We shall extend the Barnes' lemma for the case $0<p<1$.

Let $d\in \mathbb N$, and let $\omega$ be an admissible weight on $\mathbb Z^d$ satisfying weak growth condition, i.e., there is a constant $C>0$ and there is $0<\delta\leq 1$ such that $\omega(x)\geq C(1+|x|)^\delta$ for all $x$. We consider the $p$-Banach $\ast$-algebra $\mathcal A_{p\omega}$ of infinite matrices $A=(a_{kl})_{k,l \in \mathbb Z^d}$ satisfying $$\|A\|_{p\omega}=\max\Big\{\sup_{k \in \mathbb Z^d}\sum_{l \in \mathbb Z^d}|a_{kl}|^p \omega(k-l)^p, \sup_{l \in \mathbb Z^d}\sum_{k \in \mathbb Z^d}|a_{kl}|^p \omega(k-l)^p\Big\}<\infty.$$ If $A \in \mathcal A_{p\omega}$, then it defines a bounded linear operator on $\ell^2(\mathbb Z^d)$. We show that $\mathcal A_{p\omega}$ is inverse closed in $B(\ell^2(\mathbb Z^d))$.

Let $0<p\leq 1$, $d \in \mathbb N$, and let $\omega$ be an admissible weight on $\mathbb Z^{2d}$ satisfying the weak growth condition. Let $\ell^p(\mathbb Z^{2d},\omega)$ be the collection of all sequences $a=(a_{kl})_{k,l \in \mathbb Z^d}$ satisfying $\|a\|=\sum_{k,l \in \mathbb Z^d}|a_{kl}|^p\omega(k-l)^p<\infty$. Let $\theta >0$. The twisted convolution of two sequences $a=(a_{kl})_{k,l \in \mathbb Z^d}$ and $b=(b_{kl})_{k,l \in \mathbb Z^d}$ in $\ell^p(\mathbb Z^{2d},\omega)$ is given by $$(a\star_\theta b)(m,n) = \sum_{k,l\in\mathbb{Z}^d} a_{kl}b_{m-k,n-l}e^{2\pi i\theta(m-k)\cdot l}.$$ Then $\ell^p(\mathbb Z^{2d},\omega)$ is a $p$-Banach $\ast$-algebra with the twisted convolution and the involution $a^\ast_{kl}=\overline{a_{-k,-l}}e^{2\pi i \theta k\cdot l}$ for $a=(a_{kl})_{k,l\in\mathbb{Z}^d}\in\ell^p(\mathbb{Z}^{2d},\omega)$. Each $a \in \ell^p(\mathbb Z^{2d},\omega)$ defines a convolution operator $L_a$ on $\ell^2(\mathbb Z^{2d})$ given by $L_a(b)=a\star_\theta b\;(b \in \ell^2(\mathbb Z^{2d}))$. We show that $L_a$ is invertible in $B(\ell^2(\mathbb Z^{2d}))$ if and only if $a$ is invertible in $\ell^p(\mathbb Z^{2d},\omega)$ and in this case, $L_a^{-1}=L_{a^{-1}}$.

A $p$-Banach $\ast$-algebra $\mathcal{A}$ is a \emph{symmetric} if $\sigma(aa^\ast)\subset [0,\infty)$ for all $a\in\mathcal{A}$ or equivalently $\sigma(a)\in\mathbb{R}$ for all $a=a^\ast\in\mathcal{A}$. Let $\mathcal A$ and $\mathcal B$ be $p$-Banach algebras, $\mathcal A \subset \mathcal B$, and let $\mathcal A$ and $\mathcal B$ have the same unit. Then $\mathcal A$ is \emph{inverse closed} (\emph{spectrally invariant}) in $\mathcal B$ if $a \in \mathcal A$ and $a^{-1}\in \mathcal B$ imply $a^{-1}\in \mathcal A$. The property of symmetry is important itself in theory of Banach algebras as symmetric Banach algebras has many properties of $C^\ast$-algebras. Even though symmetry is defined for a given algebra and inverse-closedness gives information about relation between two nested algebras, these two topics are closely related to such a extent that most of the time the symmetry of a Banach algebra $\mathcal{A}$ is shown using inverse closedness of $\mathcal{A}$ in some $C^\ast$-algebra and it is done using the Hulanicki's lemma.

With this in consideration, first we prove Hulanicki's lemma for $p$-Banach algebras in section 2. Barnes' lemma for $p$-Banach algebras is proved in section 3. In section 4, we shall apply these lemmas to prove inverse-closedness of $p$-Banach algebra of infinite matrices over $\mathbb Z^d$ in $B(\ell^2(\mathbb Z^d))$ and the inverse-closedness of the $p$-Banach algebra $\ell^p(\mathbb Z^{2d})$ with the twisted convolution in $B(\ell^2(\mathbb Z^{2d}))$.

\section{Hulanicki's lemma for $p$-Banach algebras}
The following theorem is Hulanicki's lemma \cite[Proposition 2.5]{Hu} for $p$-Banach algebras. See \cite[6.1 Proposition]{Fe} for a proof of it for Banach algebras, i.e., for the case of $p=1$.
\begin{theorem}\label{hul}
Let $0<p\leq 1$. Let $\mathcal{A}$ be a $p$-Banach $\ast$-algebra, $S$ be a $\ast$-subalgebra of $\mathcal{A}$, and let $T$ be a faithful $\ast$-representation of $\mathcal{A}$ on Hilbert space $\mathcal{H}$ satisfying $$\|T_x\|_{op}^p=\lim_{n\to\infty}\|x^n\|^{\frac{1}{n}}\quad(x=x^\ast \in S).$$ If $\mathcal{A}$ has a unit $e$, then assume in addition that $T_e=I$, the identity operator in $B(\mathcal{H})$. If $x=x^\ast\in S$, then $\sigma_\mathcal{A}(x)=\sigma(T_x).$
\end{theorem}

We shall require the following lemma.

\begin{lemma}
Let $0<p\leq 1$. Let $\mathcal A$ be a $p$-Banach $\ast$-algebra, let $\mathcal{B}$ be the $\|\cdot\|$-closure of some commutative $\ast$-subalgebra of $\mathcal{A}$, and let $T$ be a faithful $\ast$-representation of $\mathcal A$ on a Hilbert space $\mathcal H$ satisfying $\|T_x\|_{op}^p=\lim_{n\to\infty}\|x^n\|^{\frac{1}{n}}$ for all $x=x^\ast \in \mathcal B$. If $I$ is in the operator norm closure of $T(\mathcal{B})$, then there is $e\in\mathcal{B}$ such that $T_e=I$ and $\mathcal{A}$ is unital with $e$ as unit.
\end{lemma}
\begin{proof}
For all $x\in\mathcal{B}$, let $\mu(x)=\|T_x\|_{op}^p$, and let $r(x)$ be the spectral radius of $x$. Then $\mu$ and $r$ are equivalent $p$-norms on $\mathcal{B}$ as $r$ is subadditive on $\mathcal{B}$, $r(x)=\mu(x)$ for all $x=x^\ast\in\mathcal{B}$ and $\mu(x)=\mu(x^\ast),\ r(x)=r(x^\ast)$ for all $x\in\mathcal{B}$. The completion of $\mathcal{B}$ with $\mu$, $\mathcal{B}^\mu$, is a commutative $p$-$C^\ast$-algebra isomorphic to $\overline{T(\mathcal{B})}^\mu$, and by assumption $\mathcal{B}^\mu$ has unit. As $\mathcal{B}$ is dense in $\mathcal{B}^\mu$, $\mu(x)\leq\|x\| \ (x\in\mathcal{B})$ and every $\phi\in\Delta(\mathcal{B})$ can be extended to $\widetilde{\phi}\in\Delta(\mathcal{B}^\mu)$, the Gel'fand spaces of $\mathcal{B}^\mu$ and $\mathcal{B}$ are homeomorphic via the map $\widetilde{\phi}\mapsto\widetilde{\phi}_{|\mathcal{B}}$. Since the unit of $\mathcal{B}^\mu$ has the Gel'fand transform $\mathbf{1}$, there is $x\in\mathcal{B}$ such that $\|\widehat{x}-\mathbf{1}\|_\infty<\frac{1}{2}$. Since $|\widehat{x}|\geq\frac{1}{2}$ on $\Delta(\mathcal{B})$, there is a unit $e\in\mathcal{B}$ and $T_e=I$. For $a\in\mathcal{A}$, $T_{a-ae}=T_a-T_aI=0$ and $T_{a-ea}=T_a-IT_a=0$. Since $T$ is faithful, $a=ae=ea$ and so $e$ is unit of $\mathcal{A}$.
\end{proof}

\begin{proof}[Proof of Theorem \ref{hul}]
For $x=x^\ast\in S$, let $\mathcal{B}$ be a commutative $\|\cdot\|$-closed $\ast$-subalgebra of $\mathcal{A}$ containing $x$.

If $I\in\mathcal{B}^\mu$, then the facts that the spectrum of $x$ does not separate the complex plane, $\mathcal{A}$ and $\mathcal{B}$ have the same unit, and $\mathcal{B}^\mu$ and $B(\mathcal{H})$ have the same unit imply that \begin{equation}\label{star} \sigma_\mathcal{A}(x)=\sigma_\mathcal{B}(x)=\{\phi(x):\phi\in\Delta(\mathcal{B})=\Delta(\mathcal{B}^\mu)\}=\sigma_{\mathcal{B}^\mu}(x)=\sigma(T_x). \end{equation}

If $I\notin\mathcal{B}^\mu$ and $\mathcal{A}$ has no unit, then $0\in\sigma_\mathcal{A}(x)$. Since $\mathcal{B}^\mu+\mathbb{C}I\cong\mathcal{B}^\mu\oplus\mathbb{C}$ and $\mathcal{B}^\mu+\mathbb{C}I$ and $B(\mathcal{H})$ have the same unit, $0\in\sigma_{\mathcal{B}^\mu+\mathbb{C}I}(x)=\sigma(T_x)$. So, $\sigma_\mathcal{A}(x)=\sigma(T_x)$ as the case of non-zero spectral values follows from (\ref{star}).

If $I\notin\mathcal{B}^\mu$ and $\mathcal{A}$ has unit, say $e$, then $T_e=I$ and $e\notin\mathcal{B}$. Since $\mathcal{B}^\mu+\mathbb{C}e\cong\mathcal{B}^\mu\oplus\mathbb{C}$ and $\mathcal{B}^\mu+\mathbb{C}e$ and $\mathcal{A}$ have the same unit $e$, $0\in\sigma_{\mathcal{B}^\mu+\mathbb{C}e}(x)=\sigma_\mathcal{A}(x)$. Also $0\in\sigma(T_x)$ as seen above. Combining it with (\ref{star}), we have $\sigma_\mathcal{A}(x)=\sigma(T_x)$.
\end{proof}

\section{Barnes' lemma for $p$-Banach algebras}

Let $(G,d)$ be a metric space with the counting measure $\mu$. For a subset $A$ of $G$, $\chi(A)$ denote the characteristic function of $A$. For $\delta>0$, let $\Gamma[\delta]=\{(x,y) \in G \times G : d(x,y)\leq\delta \}$, and for $x\in G$, let $\Gamma_x[\delta]=\{y\in G: d(x,y)\leq\delta \}$.

\textbf{Assumption:} There are constants $C>0, b>0$ such that $\mu(\Gamma_x[\delta])\leq C\delta^b$ for all $x\in G$ and $\delta>0$.

A \emph{kernel} $K=K(x,y)$ is a complex valued function on $G\times G$. Let $0<p\leq 1$. Let $\mathcal{A}_p$ be the collection of all kernels $K(x,y)$ such that $$\|K\|_p=\max\Big\{\sup_x\sum_y|K(x,y)|^p, \sup_y\sum_x|K(x,y)|^p \Big\}<\infty.$$ Then $(\mathcal{A}_p,\|\cdot\|_p)$ is $p$-Banach $\ast$-algebra with the convolution $$(K\star J)(x,y)=\sum_zK(x,z)J(z,y) \quad (K,J\in\mathcal{A}_p)$$ and the involution $K^\ast(x,y)=\overline{K(y,x)} \quad (K\in\mathcal{A}_p)$. Indeed, if $K,J\in\mathcal{A}_p$, then
\begin{align*} \sum_x |(K\star J)(x,y)|^p = \sum_x \left|\sum_z K(x,z)J(z,y)\right|^p &\leq \sum_x \sum_z |K(x,z)|^p|J(z,y)|^p \\&\leq \|K\|_p\|J\|_p<\infty,
\end{align*}
and the same inequality follows by reversing the roles of $x$ and $y$, so we obtain $\|K\star J\|_p\leq\|K\|_p\|J\|_p$.

Let $\delta\in(0,1]$ be fixed and define a weight $\omega:G\times G \to [1,\infty)$ by $$\omega(x,y)=(1+d(x,y))^\delta \quad ((x,y)\in G\times G).$$ By $\mathcal{A}_{p\omega}$ denote the $p$-Banach $\ast$-algebra consisting all kernels $K$ with the norm $$\|K\|_{p\omega}=\max\Big\{\sup_x\sum_y|K(x,y)|^p\omega(x,y)^p, \sup_y\sum_x|K(x,y)|^p\omega(x,y)^p \Big\}<\infty$$ and involution and convolution same as that of $\mathcal{A}_p$. Let $x,y,z\in G$. Then $d(x,y)\leq d(x,z)+d(z,y)$ implies that $\omega(x,y)\leq\omega(x,z)\omega(z,y)$ and this gives $\|K\star J\|_{p\omega}\leq \|K\|_{p\omega}\|J\|_{p\omega}$.

If $p>1$, then $\mathcal{A}_p$ is a Banach space \cite[Theorem 11.5]{Jo} with the norm $$\|K\|_p=\max\Bigg\{\sup_x\left(\sum_y|K(x,y)|^p\right)^\frac{1}{p}, \sup_y\left(\sum_x|K(x,y)|^p\right)^\frac{1}{p} \Bigg\}.$$

\begin{lemma}\label{inc}
Let $0<p\leq1$. If $K\in\mathcal{A}_{p\omega}$, then $K\in\mathcal{A}_q$ for $q\geq p$.
\end{lemma}

Let $0<p\leq1$, $q\geq p$, and let $K\in\mathcal{A}_p$. Then $K$ defines a bounded linear operator $K_q$ on $\ell^q(G)$ in the following manner $$K_q(f)(x)=\sum_yK(x,y)f(y) \quad (f\in\ell^q(G)).$$ The spectrum of $K$ in $\mathcal{A}_{p\omega}$ and $\mathcal{A}_p$ are denoted by $\sigma_{p\omega}(K)$ and $\sigma_p(K)$ respectively and the corresponding spectral radii are $r_{p\omega}(K)$ and $r_p(K)$. The spectrum and spectral radius of the operator $K_q$ in $B(\ell^q(G))$ are denoted by $\sigma(K_q)$ and $r(K_q)$ respectively.

\begin{theorem}\label{rwww}
Let $0<p\leq1$, and let $K\in\mathcal{A}_{p\omega}$. Then $r_{p\omega}(K)=r_p(K)$.
\end{theorem}
\begin{proof}
Let $0<\varepsilon\leq1$. Define a weight $\omega_\varepsilon:G\times G\to[1,\infty)$ by $$\omega_\varepsilon(x,y)=(1+\varepsilon d(x,y))^\delta.$$ Since $d(x,y)\leq d(x,z)+d(z,y)$, $1\leq\omega_\varepsilon(x,y) \leq \omega_\varepsilon(x,z)\omega_\varepsilon(z,y)$. So, $\mathcal{A}_{p\omega_\varepsilon}$ is a $p$-Banach algebra with the norm $\|K\|_{p\omega_\varepsilon}=\|K\omega_\varepsilon\|_p$. As $\omega_\varepsilon \leq \omega \leq \varepsilon^{-\delta}\omega_\varepsilon$ on $G\times G$, $\|K\|_{p\omega}\leq \varepsilon^{-p\delta} \|K\|_{p\omega_\varepsilon}$ and for $n\in\mathbb{N}$, $\|K^n\|^\frac{1}{n}_{p\omega}\leq \varepsilon^{-\frac{p\delta}{n}} \|K^n\|_{p\omega_\varepsilon}^\frac{1}{n}$. This implies that \begin{equation}\label{r1} r_{p\omega}(K)\leq r_{p\omega_\varepsilon}(K)\leq \|K\|_{p\omega_\varepsilon}. \end{equation} Since $1\leq\omega_\epsilon(x,y)$ for all $x,y\in G$, \begin{align*} \sup_x \sum_y |K(x,y)|^p &\leq \sup_x \sum_y|K(x,y)|^p \omega_\varepsilon(x,y)^p \\&\leq \sup_x \sum_y|K(x,y)|^p (1+\varepsilon^{p\delta} d(x,y)^{p\delta}) \quad (\text{as} \ 0<p\delta\leq1) \\&\leq \sup_x \sum_y|K(x,y)|^p + \varepsilon^{p\delta} \sup_x \sum_y|K(x,y)|^p d(x,y)^{p\delta}. \end{align*} Since the same inequality holds with $x$ and $y$ interchanged, $$\lim_{\varepsilon\to0}\|K\|_{p\omega_\varepsilon}=\|K\|_p.$$ This along with (\ref{r1}) gives $r_{p\omega}(K)\leq\|K\|_p$. But then $$r_{p\omega}(K)^n=r_{p\omega}(K^n)\leq\|K^n\|_p$$ and so $r_{p\omega}(K)\leq r_p(K)$. Since $\mathcal{A}_{p\omega}\subset\mathcal{A}_p$, $r_p(K)\leq r_{p\omega}(K)$. Hence, $r_{p\omega}(K)=r_p(K)$.
\end{proof}

Following lemma is a generalization of \cite[Lemma 4.4.6]{Ri} for $p$-Banach algebras.

\begin{lemma}\label{ri}
Let $0<p\leq1$, $a\mapsto T_a$ be a continuous $\ast$-representation of a $p$-normed $\ast$-algebra $\mathcal{A}$ on a Hilbert space $\mathcal{H}$, and let $a=a^\ast\in\mathcal{A}$. Then $\|T_a\|_{op}^p\leq r(a)$, where $r(a)$ is spectral radius of $a$ in $\mathcal{A}$.
\end{lemma}
\begin{proof}
Since the representation is continuous, there is some constant $C\geq1$ such that $\|T_x\|_{op}\leq C$ for all $x\in\mathcal{A}$ with $\|x\|\leq1$. Let $x\in\mathcal{A}$. If $x\neq0$, then $$\bigg\|\frac{x}{\|x\|^\frac{1}{p}}\bigg\|=1 \ \text{and so} \ \bigg\|T_{\frac{x}{\|x\|^\frac{1}{p}}}\bigg\|_{op}\leq C.$$ This gives $\|T_x\|_{op}^p\leq C^p \|x\| \leq C\|x\|$. If $x=0$, then it is trivial. Thus $\|T_x\|_{op}^p\leq C\|x\|$ for all $x\in\mathcal{A}$. Let $n\in\mathbb{N}$. Then $\|T_a^\ast T_a\|_{op}=\|T_{a^\ast}T_a\|_{op}=\|T_a\|_{op}^2$, and so $\|T_a\|_{op}^{np}=\|T_{a^n}\|_{op}^p\leq C \|a^n\|$. Thus $\|T_a\|_{op}^p\leq C^\frac{1}{n} \|a^n\|^\frac{1}{n}$. Letting  $n\to\infty$, we get $\|T_a\|_{op}^p\leq r(a)$.
\end{proof}

The generalization of Barnes' lemma \cite[Theorem 4.7]{Ba} for $p$-Banach algebras is the next theorem.

\begin{theorem}\label{ba}
Let $0<p\leq1$. If $K=K^\ast\in\mathcal{A}_{p\omega}$, then $\sigma_{p\omega}(K)=\sigma(K_2)$.
\end{theorem}
\begin{proof}
By Lemma \ref{inc}, $K\in\mathcal{A}_p$. Let $n\in\mathbb{N}$. Then \begin{equation}\label{ba1} \|K^{n+1}\|_p \leq \|\chi(\Gamma[2^n])K^{n+1}\|_p + \|\chi(\Gamma[2^n]^c)K^{n+1}\|_p, \end{equation}
where $\Gamma[2^n]^c$ is complement of the set $\Gamma[2^n]$ in $G\times G$. Since $2^{n\delta}\leq\omega(x,y)$ for all $(x,y)\in\Gamma[2^n]^c$, \begin{equation}\label{ba2} \|\chi(\Gamma[2^n]^c)K^{n+1}\|_p \leq \|K^{n+1}\|_{p\omega} 2^{-np\delta}. \end{equation} Choose $m\in\mathbb{N}$ such that $\frac{1}{2^m}<p\leq\frac{1}{2^{m-1}}$. Then $1<2^mp$ and so $\|K^{n+1}\|_{2^mp}\leq\|K^{n+1}\|_1$. Using it along with Holder's inequality and Assumption, we get
\begin{align*} & \ \quad \sum_x |K^{n+1}(x,y)|^p \big(\chi(\Gamma[2^n])(x,y) \big)^p \\ &\leq \left(\sum_x |K^{n+1}(x,y)|^{2p} \big(\chi(\Gamma[2^n])(x,y)\big)^p\right)^\frac{1}{2} \left(\sum_x \big(\chi(\Gamma[2^n])(x,y)\big)^p\right)^\frac{1}{2} \\ & \leq \left(\sum_x |K^{n+1}(x,y)|^{2^mp}\right)^\frac{1}{2^m} \left(\sum_x \big(\chi(\Gamma[2^n])(x,y)\big)^{2p}\right)^\frac{1}{2^m} \left(\sum_x \big(\chi(\Gamma[2^n])(x,y)\big)^p\right)^\frac{1}{2^{m-1}} \\ & \qquad \cdots \left(\sum_x \big(\chi(\Gamma[2^n])(x,y)\big)^p\right)^\frac{1}{4} \left(\sum_x \big(\chi(\Gamma[2^n])(x,y)\big)^p\right)^\frac{1}{2} \\ & \leq \|K^{n+1}\|_{2^mp}^p (C2^{nb})^\frac{1}{2^m} (C2^{nb})^\frac{1}{2^{m-1}} \cdots (C2^{nb})^\frac{1}{4} (C2^{nb})^\frac{1}{2} \\ & \leq \|K^{n+1}\|_1^p (C2^{nb})^{\sum_{i=1}^m\frac{1}{2^i}} .\end{align*}
Since similar inequality holds by changing the roles of $x$ and $y$, \begin{equation}\label{ba3} \|\chi(\Gamma[2^n])K^{n+1}\|_p \leq \|K^{n+1}\|_1^p (C2^{nb})^{\sum_{i=1}^m\frac{1}{2^i}}. \end{equation} So, by (\ref{ba1}), (\ref{ba2}) and (\ref{ba3}), \begin{align*} \|K^{n+1}\|_p^\frac{1}{n+1} \leq \|K^{n+1}\|_1^\frac{p}{n+1} (C^\frac{1}{n+1} 2^{\frac{n}{n+1}b})^{\sum_{i=1}^m\frac{1}{2^i}} + \|K^{n+1}\|_{p\omega}^\frac{1}{n+1} 2^{-\frac{n}{n+1}p\delta}. \end{align*} This gives $r_p(K)\leq r_1(K)^p (2^b)^{\sum_{i=1}^m\frac{1}{2^i}} + r_{p\omega}(K) 2^{-p\delta}$. By Theorem \ref{rwww}, $r_p(K)=r_{p\omega}(K)$ and so $$r_p(K) \leq r_1(K)^p \frac{(2^b)^{\sum_{i=1}^m\frac{1}{2^i}}}{1-2^{-p\delta}}.$$ Now, $$r_p(K)=r_p(K^n)^\frac{1}{n} \leq r_1(K^n)^\frac{p}{n} \left(\frac{(2^b)^{\sum_{i=1}^m\frac{1}{2^i}}}{1-2^{-p\delta}}\right)^\frac{1}{n} = r_1(K)^p \left(\frac{(2^b)^{\sum_{i=1}^m\frac{1}{2^i}}}{1-2^{-p\delta}}\right)^\frac{1}{n}. $$ Letting $n\to\infty$, we get $r_p(K) \leq r_1(K)^p$. By \cite[Theorem 4.7]{Ba}, $r_1(K)\leq \|K_2\|_{op}$ and thus $r_p(K) \leq \|K_2\|_{op}^p$. Combining it with Lemma \ref{ri} and Theorem \ref{rwww}, we get $$r_{p\omega}(K)=r_p(K)=\|K_2\|_{op}^p.$$
The result follows from Theorem \ref{hul}.
\end{proof}

\section{Inverse-closedness of some $p$-Banach algebras}
\subsection{Inverse-closedness of $p$-Beurling algebras of infinite matrices}
A weight $\omega$ on $\mathbb{R}^d$ is a non-negative measurable function satisfying $$ \omega(x+y)\leq\omega(x)\omega(y) \quad (x,y\in\mathbb{R}^d).$$ Following \cite{Gr}, we impose the following conditions on weight $\omega$ to study decay conditions of infinite matrices:
\begin{enumerate}
	\item Let $|\cdot|$ be a norm on $\mathbb{R}^d$, and let $\rho:[0,\infty)\to[0,\infty)$ be a continuous concave function such that $\rho(0)=0$. We take $\omega$ to be of the form $$\omega(x)=e^{\rho(|x|)} \quad (x\in\mathbb{R}^d).$$ Then $\omega(0)=1$ and $\omega$ is even, i.e., $\omega(x)=\omega(-x)$.
	\item $\omega$ satisfies the GRS-condition (Gel'fand-Raikov-Shilov condition \cite{Ge}) $$\lim_{n\to\infty} \omega(nx)^\frac{1}{n}=1 \quad \text{for all} \ x\in\mathbb{R}^d.$$
\end{enumerate}
The condition (ii) implies that $\lim_{\alpha\to\infty}\frac{\rho(\alpha)}{\alpha}=0$ and such a weight is called an \emph{admissible weight}. Here we will consider only admissible weights and that too mostly on $\mathbb{Z}^d$ which is obtained by restricting $\omega$ on $\mathbb{Z}^d$.

Let $0<p\leq1$. Let $\mathcal{A}_{p\omega}$ be the collection of all matrices $A=(a_{kl})_{k,l\in\mathbb{Z}^d}$ satisfying $$\|A\|_{p\omega}=\max \Big\{ \sup_{k\in\mathbb{Z}^d} \sum_{l\in\mathbb{Z}^d} |a_{kl}|^p \omega(k-l)^p, \sup_{l\in\mathbb{Z}^d} \sum_{k\in\mathbb{Z}^d} |a_{kl}|^p\omega(k-l)^p \Big\} <\infty.$$ Then $\mathcal{A}_{p\omega}$ is a $p$-Banach $\ast$-algebra with norm $\|\cdot\|_{p\omega}$, involution $\ast:A=(a_{kl})\mapsto A^\ast=(a^\ast_{kl})$ where $a^\ast_{kl}=\overline{a_{lk}}$ and convolution as multiplication defined by $(A\star B)_{kl}=\sum_{j\in\mathbb{Z}^d} a_{kj}b_{jl}$ for $A=(a_{kl}),B=(b_{kl})\in\mathcal{A}_{p\omega}$.

Note that we will skip writing $\mathbb{Z}^d$ in the indices as the case will be clear and $(A)_{kl}$ denote the $(k,l)^{th}$ entry of the matrix $A$. When the trivial weight $\omega\equiv1$ is in consideration, the corresponding space will be denoted by $\mathcal{A}_p$.

If $A\in\mathcal{A}_{p\omega}$, then $A\in\mathcal{A}_q$ for all $q\geq p$ and so the standard Schur test implies that $A\in B(\ell^q(\mathbb{Z}^d))$ for all $q\geq p$. So, $\mathcal{A}_{p\omega}$ can be seen as a $\ast$-subalgebra of bounded operators acting on $\ell^2(\mathbb{Z}^d)$. The spectrum of $A$ in $\mathcal{A}_{p\omega}$, $\mathcal{A}_q \ (q\geq p)$ and as an operator in $B(\ell^2(\mathbb{Z}^d))$ will be denoted by $\sigma_{p\omega}(A)$, $\sigma_q(A) $ and $\sigma(A) $ respectively and the corresponding spectral radii are denoted by $r_{p\omega}(A)$, $r_q(A)$ and $r(A)$.

A weight $\omega$ is said to be satisfying \emph{weak growth condition} if for some positive constant $C$ and $0<\delta\leq1$, $$\omega(x)\geq C(1+|x|)^\delta, \quad \text{for all}\ x.$$
Following is our main theorem in this section.

\begin{theorem}\label{grm}
Let $\omega$ be an admissible weight satisfying the weak growth condition, and let $A=A^\ast\in\mathcal{A}_{p\omega}$. Then $$r_{p\omega}(A)=\|A\|_{op}^p.$$ Consequently, $\sigma_{p\omega}(A)=\sigma(A)$ and $\mathcal{A}_{p\omega}$ is symmetric.
\end{theorem}

We write a corollary of above theorem explicitly stating property of symmetry and inverse-closedness.

\begin{corollary}\label{grc7}
Let $\omega$ be an admissible weight satisfying the weak growth condition, i.e., $\omega(x)\geq C(1+|x|)^\delta$ for some positive constant $C$ and some $\delta\in(0,1]$. If $A\in B(\ell^2(\mathbb{Z}^d))$ satisfies the weighted Schur-type condition $$\max \Big\{ \sup_{k\in\mathbb{Z}^d} \sum_{l\in\mathbb{Z}^d} |a_{kl}|^p \omega(k-l)^p, \sup_{l\in\mathbb{Z}^d} \sum_{k\in\mathbb{Z}^d} |a_{kl}|^p\omega(k-l)^p \Big\} <\infty,$$ then the inverse matrix $A^{-1}=(b_{kl})_{k,l\in\mathbb{Z}^d}$ satisfies the same Schur-type condition $$\max \Big\{ \sup_{k\in\mathbb{Z}^d} \sum_{l\in\mathbb{Z}^d} |b_{kl}|^p \omega(k-l)^p, \sup_{l\in\mathbb{Z}^d} \sum_{k\in\mathbb{Z}^d} |b_{kl}|^p\omega(k-l)^p \Big\} <\infty.$$ If in addition $A$ is a positive operator, then the matrices corresponding to $A^\alpha$ for each $\alpha\in\mathbb{R}$ are also in $\mathcal{A}_{p\omega}$.
\end{corollary}

We shall require the following two lemmas. The first one of which constructs a sequence of auxiliary weights $\omega_n$ using techniques developed in \cite{Le} and \cite{Py}.

\begin{lemma}\label{grl8}\cite[Lemma 8]{Gr}
Let $\omega$ be an unbounded admissible weight. Then there is a sequence of admissible weights $\omega_n$ such that
\begin{enumerate}
	\item $\omega_{n+1}\leq\omega_n\leq\omega$ for all $n\in\mathbb{N}$,
	\item there are constants $c_n>0$ such that $\omega\leq c_n\omega_n$, and
	\item $\lim_{n\to\infty} \omega_n=1$ uniformly on compact subsets of $\mathbb{R}^d$.
\end{enumerate}
\end{lemma}
Note that all $\omega_n$ are equivalent (by (i) and (ii)) and satisfies GRS-condition (by (i)). So, $\mathcal{A}_{p\omega}$ and $\mathcal{A}_{p\omega_n}$ coincides having equivalent norms and thus for all $A\in\mathcal{A}_{p\omega}$, $$r_{p\omega}(A)=r_{p\omega_n}(A) \quad (n\in\mathbb{N}).$$
We just give an idea about the construction of $\omega_n$ as it will be required. For detailed proof refer to \cite{Gr}.

\textbf{Construction of $\omega_n$:} For $n\in\mathbb{N}$, let $$\gamma_n=\sup_{\mu\geq\rho^{-1}(n)} \frac{\rho(\mu)-n}{\mu} >0.$$ Since $\rho$ is continuous and $\lim_{n\to\infty} \frac{\rho(\mu)-n}{\mu}=0$, there is some $\beta_n\geq\rho^{-1}(n)$ such that $$\gamma_n=\frac{\rho(\beta_n)-n}{\beta_n}.$$ Define $\rho_n:[0,\infty)\to[0,\infty)$ by \begin{align*} \rho_n(\mu)=\begin{cases} \gamma_n\mu, \ &0\leq\mu\leq\beta_n, \\ \rho(\mu)-n, \ &\mu\geq\beta_n. \end{cases} \end{align*}
Define corresponding weight $\omega_n$ by $$\omega_n(x)=e^{\rho_n(|x|)} \quad (x\in\mathbb{R}^d).$$

\begin{lemma}\label{grl9}
With the assumptions of Theorem $\ref{grm}$ and $\omega_n$ as in Lemma $\ref{grl8}$, for every $A=A^\ast\in\mathcal{A}_{p\omega}$, $$\lim_{n\to\infty} \|A\|_{p\omega_n}=\|A\|_{p}$$ and \begin{equation}\label{gr23} r_{p\omega}(A)=r_p(A)=\|A\|_{op}^p. \end{equation}
\end{lemma}
\begin{proof}
Let $\epsilon>0$. Let $A=A^\ast\in\mathcal{A}_{p\omega}$. Then $$\|A\|_{p\omega_n}=\sup_k\sum_l|a_{kl}|^p\omega_n(k-l)^p<\infty.$$ By construction of $\omega_n$, $\omega_n(x)=e^{-n}\omega(x)$ for all $|x|\geq\beta_n$. So, there is $n_0\in\mathbb{N}$ such that $$\sup_k\sum_{l:|k-l|\geq\beta_{n_0}}|a_{kl}|^p\omega_{n_0}(k-l)^p\leq e^{-pn_0}\|A\|_{p\omega}<\epsilon.$$ Since $\omega_{n+1}\leq\omega_n\leq\omega$ for all $n$, if $n\geq n_0$, then $$\sup_k\sum_{l:|k-l|\geq\beta_{n_0}}|a_{kl}|^p\omega_n(k-l)^p<\epsilon.$$ Now, if $|x|\leq\beta_{n_0}$, then $\omega_n\to1$ uniformly and so there is $n_1\in\mathbb{N}$ such that for $n\geq n_1$, $$\sup_k \sum_{l:|k-l|\leq\beta_{n_0}} |a_{kl}|^p\omega_n(k-l)^p \leq (1+\epsilon^p) \sup_k \sum_l |a_{kl}|^p. $$ So, we have $$\|A\|_{p\omega_n} \leq \epsilon + (1+\epsilon^p) \|A\|_p.$$ Thus, $\lim_{n\to\infty} \|A\|_{p\omega_n}\leq\|A\|_p$. Since $\omega_n\geq1$, reverse inequality always holds.

Since $\omega$ and $\omega_n$ are equivalent weights for all $n\in\mathbb{N}$, $$r_{p\omega}(A)^k=r_{p\omega}(A^k)=r_{p\omega_n}(A^k)\leq\|A^k\|_{p\omega_n} \quad (k,n\in\mathbb{N}).$$ So, $$r_{p\omega}(A)^k\leq\lim_{n\to\infty}\|A^k\|_{p\omega_n}=\|A^k\|_p \quad (k\in\mathbb{N})$$ and this gives $r_{p\omega}(A)\leq r_p(A)$. Since $\mathcal{A}_{p\omega}\subset\mathcal{A}_p$, $r_p(A)\leq r_{p\omega}(A)$ is always true. Now, as $\omega(x)\geq C(1+|x|)^\delta=\tau_\delta(x)$ and $0<\delta\leq1$, $\mathcal{A}_{p\omega}\subset\mathcal{A}_{p\tau_\delta}$, and so by Theorem \ref{ba}, $r_{p}(A)=\|A\|_{op}^p$. This completes the proof.
\end{proof}

\begin{proof}[Proof of Theorem $\ref{grm}$ and Corollary $\ref{grc7}$]
Combining Theorem \ref{hul} with (\ref{gr23}), we get $\sigma_{p\omega}(A)=\sigma(A)$ for all $A\in\mathcal{A}_{p\omega}$ and the symmetry of $\mathcal{A}_{p\omega}$ follows.

Now, if $A\in\mathcal{A}_{p\omega}$ is an invertible positive operator in $B(\ell^2(\mathbb{Z}^d))$, then $\sigma(A)\subset[\delta,\infty)$ for some $\delta>0$ and it follows that $\sigma_{p\omega}(A)\subset[\delta,\infty)$. The theorem follows from Riesz functional calculus (see \cite{Ru} and \cite{Ze}).
\end{proof}

\subsection{Wiener's Lemma for Twisted Convolution}

\begin{definition}\cite{Le}
Let $\theta>0$. The \emph{twisted convolution} of two sequences $a=(a_{kl})_{k,l\in\mathbb{Z}^d}$ and  $b=(b_{kl})_{k,l\in\mathbb{Z}^d}$ is defined as \begin{equation}\label{gr29} (a\star_\theta b)(m,n) = \sum_{k,l\in\mathbb{Z}^d} a_{kl}b_{m-k,n-l}e^{2\pi i\theta(m-k)\cdot l} = \sum_{k,l\in\mathbb{Z}^d} a_{m-k,n-l}b_{kl}e^{2\pi i\theta k\cdot(n-l)}.\end{equation}
\end{definition}

Let $0<p\leq1$, and let $q\geq 1$. Since $$\|a\star_\theta b\|_q \leq \||a|\star|b|\|_q \leq \|a\|_1\|b\|_q \leq \|a\|_p^\frac{1}{p}\|b\|_q,$$ the twisted convolution operator $L_a(b)=a\star_\theta b$ is in $B(\ell^q(\mathbb{Z}^{2d}))$ for any $a\in\ell^p(\mathbb{Z}^{2d})$.

In this section we consider the space $\ell^p(\mathbb{Z}^{2d})$ with twisted convolution as product and involution $a^\ast_{kl}=\overline{a_{-k,-l}}e^{2\pi i \theta k\cdot l}$ for $a=(a_{kl})_{k,l\in\mathbb{Z}^d}\in\ell^p(\mathbb{Z}^{2d})$.

\begin{theorem}
Let $0<p\leq1$, $\omega$ be an admissible weight satisfying weak growth condition, and let $a\in\ell^p(\mathbb{Z}^{2d},\omega)$ be such that the twisted convolution operator $L_a$ is invertible in $B(\ell^2(\mathbb{Z}^{2d}))$. Then $a$ is invertible in $\ell^p(\mathbb{Z}^{2d},\omega)$ and $L_a^{-1}=L_b$ for some $b\in\ell^p(\mathbb{Z}^{2d},\omega)$.
\end{theorem}
\begin{proof}
For $L_a\in B(\ell^2(\mathbb{Z}^{2d}))$, by (\ref{gr29}), the matrix $A$ associated with it has the entries $$A_{(k,l),(m,n)}=a_{m-k,n-l}e^{2\pi i\theta k\cdot(n-l)}.$$ Now, \begin{equation}\label{gr31} \sup_{(k,l)\in\mathbb{Z}^{2d}} \sum_{(m,n)\in\mathbb{Z}^{2d}} |A_{(k,l),(m,n)}|^p \omega(k-m,l-n)^p=\|a\|_{p\omega}<\infty, \end{equation} and likewise with index interchanged. This gives $\|A\|_{p\omega}=\|a\|_{p\omega}$ and $A\in\mathcal{A}_{p\omega}$. By Theorem \ref{grm}, $B=A^{-1}\in\mathcal{A}_{p\omega}$. So, it remains to show that there is some $b\in\ell^p(\mathbb{Z}^{2d},\omega)$ such that $B=L_b$. Let $b\in\ell^2(\mathbb{Z}^{2d})$ be such that $L_ab=\delta_0$ where $\delta_0(0)=1$ and $\delta_0(m)=0$ for non-zero $m\in\mathbb{Z}^{2d}$. Let $c\in c_{00}=\{d=(d_{kl})_{k,l\in\mathbb{Z}^d}: \text{supp(d) is finite}\}$. Then $$L_a(L_b-B)c=a\star_\theta(b\star_\theta c)-L_aL_a^{-1}c=c-c=0.$$ So, $L_b=B$ on $c_{00}$. Since $c_{00}$ is dense in $\ell^2(\mathbb{Z}^{2d})$, it follows that the matrix of $L_a$ and $B$ are same and by (\ref{gr31}), $b\in\ell^p(\mathbb{Z}^{2d},\omega)$. The rest follows.
\end{proof}

\bibliographystyle{amsplain}

\end{document}